\documentclass[12pt]{amsart}
\usepackage{amsmath,amssymb,amsbsy,amsfonts,amsthm,latexsym,mathabx,
            amsopn,amstext,amsxtra,euscript,amscd,stmaryrd,mathrsfs,
            cite,array,mathtools,enumerate}

\usepackage{url}
\usepackage[colorlinks,linkcolor=blue,anchorcolor=blue,citecolor=blue,backref=page]{hyperref}

\usepackage{color}
\usepackage{a4wide}

\usepackage{color}
\usepackage{float}

\hypersetup{breaklinks=true}

\usepackage[english]{babel}

\usepackage{mathtools}
\usepackage{todonotes}

\usepackage[norefs,nocites]{refcheck}

\usepackage{enumitem}

\allowdisplaybreaks

\usepackage{todonotes}
\usepackage{url}
\usepackage[colorlinks,linkcolor=blue,anchorcolor=blue,citecolor=blue,backref=page]{hyperref}
\usepackage[noabbrev,capitalize]{cleveref}

\usepackage{tikz}
\usetikzlibrary{arrows.meta}
\usepackage{tikz-cd}

\newtheorem{theorem}{Theorem}
\newtheorem{lemma}[theorem]{Lemma}

\newtheorem{prop}[theorem]{Proposition}

\numberwithin{equation}{section}
\numberwithin{theorem}{section}

 \newcommand{\F}{\mathbb{F}}
\newcommand{\K}{\mathbb{K}}

\newcommand{\U}{U}

\newcommand{\X}{Z}

\DeclareMathOperator{\Char}{char}

\newcommand{\Fq}{\mathbb{F}_q}
\newcommand{\Fqb}{\overline{\mathbb{F}}_q}

\def\cN{{\mathcal N}}
\def\cO{{\mathcal O}}

\title[Linear complexity of  hyperelliptic curve sequences]{Linear complexity of some sequences derived from hyperelliptic curves of genus 2}

\author[V.~Anupindi]{Vishnupriya Anupindi}
 \author[L.~M{\'e}rai]{L{\'a}szl{\'o} M{\'e}rai}
 \address{Johann Radon Institute for Computational and Applied Mathematics, Austrian Academy of Sciences and Institute of Financial Mathematics and Applied Number Theory,
 Johannes Kepler University,  Altenberger Stra\ss e 69, A-4040 Linz, Austria} 
 \email{vishnupriya.anupindi@oeaw.ac.at}
 \email{laszlo.merai@oeaw.ac.at}

\begin{document}

\begin{abstract}
For a given hyperelliptic curve $C$ over a finite field with Jacobian $J_C$, we consider the hyperelliptic analogue of the congruential generator defined by  $W_n=W_{n-1}+D$ for $n\geq 1$ and $D,W_0\in J_C$. We show that curves of genus 2 produce sequences with large linear complexity.
\end{abstract}

\keywords{elliptic curve, hyperelliptic curve, linear complexity, congruential generator}

\subjclass[2020]{11G05, 11G20, 11K45, 11T71 }
\maketitle

\section{Introduction}

Let $\Fq$ be a finite field with characteristic $p\geq 3$ and consider the hyperelliptic curve $C$ of genus $g\geq 1$ defined by
\begin{equation} \label{def:hec}
    C: Y^2=f(X), 
\end{equation}
where $f(X)\in \Fq[X]$ is a monic polynomial of degree $2g+1$.

For details on hyperelliptic curves, see \cite{cohen2005handbook,Galbraith-book, Koblitz-book}.

We denote the $\Fq$-rational points of $C$ by $C(\Fq)$, which are the solutions over $\Fq$ of the defining equation \eqref{def:hec} together with a point $\cO$ at infinity. 
By the Hasse-Weil bound \cite[Theorem~5.2.3]{stichtenothAlgebraicFunctionFields2009}, we have
\begin{equation}\label{eq:number_of_points}
    ||C(\Fq)| -(q+1) |\leq 2g q^{1/2}.
\end{equation}

Unlike elliptic curves (curves with genus $g=1$), the points of a hyperelliptic curve with higher genus ($g\geq 2$) do not form an additive group. However, one can define a group operation by introducing the \emph{Jacobian} $J_C$ of the curve $C$.    

For an affine point $P=(x,y)\in C$, we write $-P=(x,-y)$ (and $-\cO=\cO$). A \emph{divisor} $D$ of $C(\Fqb)$ is an element of the free abelian group over the points of $C(\Fqb)$, e.\ g.\ $D =\sum_{P\in C(\Fqb)} n_P P$  with $n_P\in \mathbb{Z}$ and $n_P = 0$ for almost all points $P$.

Then any element $D\in J_C$ of the Jacobian  can be uniquely represented as a \emph{reduced divisor}
$$
\psi(D)=P_1+\dots +P_r-r\cO,
$$
where $1\leq r \leq g$, $\cO$ is the point of $C$ at infinity, $P_1, \dots, P_r\in C$, $P_i\neq \cO, 1\leq i \leq r$ and $P_i\neq -P_j$, $1\leq i<j\leq r$. The element $D$ is said to be defined over $\Fq$ if the \emph{Frobenius endomorphism}, defined by $\sigma((x,y))=(x^q,y^q)$ permutes the set $\{P_1,\dots, P_r\}$. We use $J_C(\Fq)$ to denote the set of elements of $J_C$ which are defined over $\Fq$. It follows from the Hasse-Weil Theorem   \cite[Theorem 5.1.15 and 5.2.1]{stichtenothAlgebraicFunctionFields2009}, that
\begin{equation}\label{eq:size_of_jac}
    (q^{1/2}-1)^{2g}\leq |J_{C}(\Fq)|\leq (q^{1/2}+1)^{2g}.
\end{equation}
It is common to represent the elements of the Jacobian by the \emph{Mumford representation} \cite{mumfordTataLecturesTheta2007a}. Let $D \in J_C(\Fq)$, then the Mumford representation $\eta(D)$ is a pair $[u,v]$ of polynomials such that 
	\renewcommand{\theenumi}{\alph{enumi}}
	\begin{enumerate}
		\item \label{mum-a} $u$ is monic,
		\item $u$ divides $f-v^2$,
		\item $\deg(v)< \deg(u)\leq g$.
	\end{enumerate}
Let $\psi(D)=\sum_{i=1}^r P_i-r \cO$, where $P_1, \dots, P_r\in C$, $P_i\neq \cO, 1\leq i \leq r$ and $P_i\neq -P_j$, $1\leq i<j\leq r$. Put $P_i=(x_i,y_i)$. Then $u=\prod_{i=1}^r(X-x_i)$ and $v(x_i)=y_i$ is of the same multiplicity as $P_i$ in $\psi(D)$.

One can define a group operation, denoted by $+$, on the Jacobian $J_C$. 
In Mumford coordinates, the group operation can be computed using Cantor's algorithm~\cite{cantor}. This algorithm can be made highly effective for small genus, which is the most important case in cryptographic applications.  See for example \cite{langeFormulaeArithmeticGenus2005} for the case of genus $g=2$.

\bigskip

In this paper, we study the properties of pseudorandomness of certain walks on the Jacobian $J_C(\Fq)$. Namely, let $D\in J_C(\Fq)$ and define the sequence
\begin{equation}\label{eq:S}
    W_n=D + W_{n-1}=nD + W_0, \quad n=1,2,\dots,
\end{equation}
with some initial value $W_0\in J_C(\Fq)$.

In the special case $g=1$, that is, when $C$ is an elliptic curve and $J_C\cong C$, the sequence $(W_n)$ has been suggested as a pseudorandom number generator in \cite{hallgren} and later many pseudorandom properties of this sequence have been studied \cite{BeelenDoumen,ElMaShp2001,GongBersonStinson99,GongLam2001,meraiLinearComplexityProfile2015,hessLinearComplexityMultidimensional2005,TopWinterhof2007,MeraiECAdditive,MeraiEcMult,ChenECLegendre}. 

In particular, the linear complexity of the coordinates of $(W_n)$ has been studied \cite{hessLinearComplexityMultidimensional2005, meraiLinearComplexityProfile2015,TopWinterhof2007}.
%\commL{$w\rightarrow a$ here}
We recall, that the \textit{$N$-th linear complexity} $L(s_n,N)$ of a sequence $(s_n)$ over the finite field $\Fq$ is defined as the smallest non-negative integer $ L $ such that the first $ N $ terms of the sequence $ (s_n) $ can be generated by a linear recurrence relation over $ \Fq $ of order $ L $, i.e.\ there exist $ c_0, c_1, \dots , c_{L-1} \in \Fq $ such that 
\begin{equation*}
	s_{n+L} = c_0s_{n} + c_1s_{n+1} + \dots  + c_{L-1}s_{n+L-1},\quad  0 \leq n \leq N-L-1.
\end{equation*}
The $N$-th linear complexity measures the unpredictability of a sequence and thus is an important figure of merit in cryptography. Clearly, large linear complexity is a desired (but not sufficient) property for such applications.  For more details, see \cite{MeidlWinterhof,Nied1, Winterhof2010}.

Hess and Shparlinski \cite{hessLinearComplexityMultidimensional2005} estimated the linear complexity of the coordinates for the elliptic curve analog of \eqref{eq:S}. Namely, let $x(\cdot)$ and $y(\cdot)$ be the coordinate functions of the curve such that for any affine point $P=(x(P),y(P))$. Among others, Hess and Shparlinski \cite{hessLinearComplexityMultidimensional2005} proved
$$
L(x(W_n ),N)\geq \min\left\{\frac{N}{4},\frac{t}{3}\right\},
$$
where $t$ is order of $D$, see also \cite{meraiLinearComplexityProfile2015,TopWinterhof2007}. 

\bigskip

In this paper, we estimate the linear complexity of the Mumford coordinates of the sequence $(W_n)$ defined by \eqref{eq:S} for hyperelliptic curves of genus $g=2$. More precisely, for $D\in J_C(\Fq)$, let $\eta(D)=[u(D),v(D)]$ be its Mumford representation and write
\begin{equation}\label{eq:coordinates}
    u(D)=u_2(D)X^2+u_1(D)X+u_0(D) \quad \text{and} \quad      v(D)=v_1(D)X+v_0(D), 
\end{equation}
where $u_2(D),u_1(D),u_0(D),v_1(D),v_0(D) \in\Fq$. By    \eqref{mum-a}, $u(D)$ is monic and thus $u_2(D)=1$  for all but approximately $q$ values $D$ and otherwise $u_2(D) = 0$. Thus, one cannot expect strong randomness properties of it. 
However, for the other coefficients, our result implies the lower bound for the linear complexity
\begin{equation}\label{eq:example}
L(u_0(W_n),N)\geq \left\lfloor c
		    \frac{\min\{t,N\}}{  q }\right\rfloor, 
\end{equation}
for some absolute and explicit constant $c>0$, where $t$ is the order of $D$, see Theorem~\ref{thm:main}. A similar bound holds for the other coefficients in \eqref{eq:coordinates} (except for $u_2$).

The most promising case is when the Jacobian $J_C(\Fq)$ is close to being a cyclic group, and $D$ has order $t=q^{2+o(1)}$ (cf. \eqref{eq:size_of_jac}).

\bigskip

In Section~\ref{sec:HEC}, we recall and prove the necessary tools concerning the arithmetic of hyperelliptic curves. In particular, we recall the Grant representation \cite{grant1990formal} of the Jacobian of hyperelliptic curves of genus $g=2$ as it gives explicit formulas for the addition law. In Section~\ref{sec:preliminaries}, we recall other required preliminaries. Finally, in Section~\ref{sec:main}, we state and prove our main result.

\section{Arithmetic of hyperelliptic curve with genus \texorpdfstring{$g=2$}{g=2}}
\label{sec:HEC} 

Let $C$ be the hyperelliptic curve defined by \eqref{def:hec} with
$$
f(X) = X^5 + b_1X^4 + b_2X^3 + b_3X^2 + b_4X + b_5 \in \Fq[X],
$$
for the finite field $\Fq$ with characteristic $\Char(\Fq) \neq 2$. Let $\Fqb$ be the algebraic closure of $\Fq$.

\subsection{Grant representation}
The Jacobian $J_C$ is an abelian variety of dimension $2$  \cite[Theorem~A.8.1.1]{hindryDiophantineGeometryIntroduction2000}. In \cite{grant1990formal}, Grant provides an embedding of $J_C$ into the projective space $\mathbb{P}^8$.
Namely, the Jacobian $J_C$ can be identified with the vanishing locus of $13$ homogenous polynomials,
\begin{equation*}
    J_C \cong V(f_2^h,\dots,f_{14}^h) = \{ z \in \mathbb{P}^8 : f_i^h(z) = 0, 2 \leq i \leq 14 \} ,
\end{equation*}
where $ f_i^h \in \Fq[\X_0, \X_{11},\X_{12},\X_{22},\X_{111},\X_{112},\X_{122},\X_{222},\X]$ are homogenized with respect to the variable $\X_0$.
For the expressions $ f_i^h$, see \cref{sec:def_eq_J}.

Let 
\begin{equation}\label{eq:Theta}
\Theta(\Fq)=\{D\in J_C(\Fq): \psi(D)=P-\cO, P\in C(\Fq) \}
\end{equation}
be the preimage of $C(\Fq) - \cO$ under $\psi$. Also write $\Theta=\Theta(\Fqb)$.
Let $\iota : J_C \rightarrow \mathbb{P}^8$ be the embedding, then
\begin{equation}\label{eq:iota}
    \iota(D) = 
        \begin{cases}
            (1:z_{11}:z_{12}:z_{22}:z_{111}:z_{112}:z_{122}:z_{222}:z)  & \text{if~} D \in J_C\setminus \Theta, \\ 
            (0:0:0:0:1:0:0:0:0) & \text{if~} D = \cO ,\\
            (0:0:0:0:-x^3: - x^2 : -x : 1 : -y) & \text{if~} D \in \Theta \setminus\cO.
            
        \end{cases}
\end{equation}

For $D\in J_C(\Fq)\setminus \Theta(\Fq)$, the components $z_{ij}, z_{ijk}$ of $\iota(D)$ can be expressed as rational functions in the coordinates $(x_1,x_2,y_1,y_2)$ of $\psi(D) = (x_1,y_1) + (x_2,y_2) - 2\cO$.
Moreover, the Mumford representation
$\eta(D)= [u,v] = [X^2 + u_1X + u_0, v_1X + v_0]$ can also be expressed with respect to $(x_1,x_2,y_1,y_2)$.
Namely, write $\psi(D) =(x_1,y_1) + (x_2,y_2) - 2\cO$ ,
if $\eta(D) = [u,v] = [X^2 + u_1X + u_0, v_1X + v_0]$, then the relation between $\iota (D)$ and $\eta(D)$ is given as follows: 

For the case $x_1 \neq x_2 $, we have 
        \begin{align}\label{eq:repr}
          z_{12} &= -x_1x_2 =-  u_0 ,  &z_{22} &= x_1+x_2 =-u_1  , \\ \notag
          z_{122}  &= \frac{x_1y_2- x_2y_1}{x_1-x_2} =  v_0 , &z_{222}& = \frac{y_1-y_2}{x_1-x_2} = v_1,  
        \end{align}
        \begin{align*}
        	z_{11}= &\frac{(x_1 + x_2)(x_1x_2)^2 + 2b_1(x_1x_2)^2 + b_2(x_1 + x_2)x_1x_2 + 2b_3x_1x_2}{(x_1-x_2)^2}\\
        	& +\frac{b_4(x_1 + x_2) + 2b_5 - 2y_1y_2}{(x_1-x_2)^2}.
        \end{align*} 
For the case $x_1 = x_2$ and $y_1 = y_2\neq 0$, we consider the identifications on page 109 of \cite{grant1990formal} and after performing some elementary computations, we obtain:
        \begin{align*}
          z_{12} &= -x_1^2 =  - u_0, & z_{22} &= 2x_1 = -u_1  \\
            z_{122} &= -\frac{f'(x_1)}{2y_1}x_1 + y_1 = v_0 , & z_{222}& = \frac{f'(x_1)}{2y_1} = v_1,  \\
            z_{11} &= \Big( \frac{f'(x_1)}{2y_1} \Big)^2 - 6x_1^3 - 4b_1x_1^2 - 2b_2x_1 - b_3
        \end{align*}
For the case $x_1 = x_2$ and $y_1 = -y_2$, we obtain $\eta(D) = [u,v] = [1,0]$ and $\iota(D) = (0:0:0:0:1:0:0:0:0)$.

The coordinates $z_{111},z_{112}$ and $z$ are given by the defining polynomials $f_4, f_3$ and $f_2$ respectively. See \cref{app:rational function}.

We denote the affine part of $J_C$ under $\iota$ with respect to $\X_0$ by $U$. Then 
$$
U = J_C\setminus \Theta.
$$
%Let $$\Fq[\mathbf{\X}] = \Fq[\X_{11},\X_{12},\X_{22},\X_{111},\X_{112},\X_{122},\X_{222},\X] $$
%
Moreover, 
by \cite[Corollary~2.15]{grant1990formal}, we have
$$
\U=V(f_2,\dots, f_7), f_i \in
\Fq[\mathbf{\X}] = \Fq[\X_{11},\X_{12},\X_{22},\X_{111},\X_{112},\X_{122},\X_{222},\X]
.
$$
Since $ J_C$ 
is a variety, it is irreducible. It follows that $ U $
is irreducible and dense, see \cite[Example 1.1.3 ]{hartshorneAlgebraicGeometry1977}. Since $J_C$ has dimension $2$, the dimension of $ U $ is also $2$.

Let $ \Fq(U) $ be the function field of $ \U $. Since $ \U $ is a dense open subset of $J_C$, one can show that $ \Fq(U) = \Fq(J_C)$, see \cite[Theorem~3.4]{hartshorneAlgebraicGeometry1977}.

Let $ h \in \Fq(U) $ be a rational function. Since $h$ is an equivalence class, we can choose a representative element $\frac{h_1}{h_2}$, where $\deg h_1$ is minimal and we define
$$
\deg h =\max\{\deg h_1, \deg h_2\}.
$$

\subsection{Group law}
We describe the group law $(Q,R) \mapsto Q+R$ for the most common case, that is, if all elements $Q,R,Q+R,Q-R$ belong to $\U$, see  \cite[Theorem 3.3]{grant1990formal}.

\begin{lemma}\label{lemma:addition_grant}
Assume that $Q,R,Q+R,Q-R \in U$.
We write
  \begin{equation}\label{eq:q}
  q(Q,R)=z_{11}(Q)-z_{11}(R)+z_{12}(Q)z_{22}(R)-z_{12}(R)z_{22}(Q).
  \end{equation}
Then there are rational functions $q_1,q_2,q_{11},q_{12},q_{22},q_{111} $ on $U\times U$ such that for $1\leq i\leq j\leq 2$ we have
\begin{equation}\label{eq:add1}
       z_{ij}(Q+R) = -z_{ij}(Q)-z_{ij}(R) + \frac{1}{4}\left( \frac{q_{i}(Q,R)}{q(Q,R)} \right) \left( \frac{q_{j}(Q,R)}{q(Q,R)} \right) -\frac{1}{4} \left( \frac{q_{ij}(Q,R)}{q(Q,R)} \right)
\end{equation} 
 and
 \begin{equation}\label{eq:add2}
    \begin{split}
         z_{111}(Q+R)= &-\frac{1}{2}z_{111}(Q)-\frac{1}{2}z_{111}(R) + \frac{3}{16}\frac{q_{1}(Q,R)q_{11}(Q,R)}{q(Q,R)^2}-\frac{1}{16}\frac{q_{111}(Q,R)}{q(Q,R)} \\
     &- \frac{1}{8} \left( \frac{q_{1}(Q,R)}{q(Q,R)} \right) ^3 
     + \frac{3}{4}(z_{11}(Q)+z_{11}(R))\frac{q_{1}(Q,R)}{q(Q,R)}, \\
    z(Q+R) = &\frac{1}{2} \left( z_{11}(Q+R)z_{22}(Q+R) - z_{12}^2(Q+R) + b_2z_{12}(Q+R) -b_4 \right).
    \end{split}
\end{equation}
\end{lemma}
 The definitions of the rational functions $q, q_1,q_2,q_{11},q_{12},q_{22},q_{111} $ are in \cref{sec:addition_formulas}. We also computed formulas for $z_{112}, z_{122}, z_{222}$ which are listed in \cref{sec:addition_formulas}.

For a fixed $ R \in \U$, we define 
 \begin{equation}\label{eq:add_fn}
 	z_{ij}^R(Q) = z_{ij}(Q+R),
 	\quad
 	z_{ijk}^R(Q) = z_{ijk}(Q+R),
 	\quad
 	z^R = z(Q+R).
 \end{equation}
If we consider $z_{ij}^R(Q), z_{ijk}^R(Q)$ to be polynomials in the variables $z_{ij}(Q), z_{ijk}(Q)$ and $z(Q)$, then it follows from \eqref{eq:add1}, \eqref{eq:add2} and \cref{sec:addition_formulas} that 
 \begin{equation} \label{eq:add_deg}
 	\deg z_{ij}^R \leq 3 \text{~  ,  ~} \deg z_{ijk}^R \leq 4 \text{~  and  ~} \deg z^R \leq 6.
 \end{equation}

 \begin{lemma}\label{lem:1}
  Let $q(Q,R)$ be defined by \eqref{eq:q} and set $q_R(Q) = q(Q,R)$.  Then for any fixed $R\in \U$,  the zero set $\{ q_R(Q)=q(Q,R)=0 \}$ has dimension one and $\Theta \pm R\subset\{q_R=0\}$. Moreover
  if $R\neq \pm R'$, then 
  \begin{equation}\label{eq:cardinality_zeros_q_1}
    |\{q_{R} = 0\} \cap \{q_{R'} = 0\} \cap \U| \leq 20.
\end{equation}
 \end{lemma}
  \begin{proof}
 Let $\bar{C}$ be the lifted curve of $C$ defined over some number field $\mathbb{K}$ and let $\mathfrak{p}$ be a prime ideal such that $C\equiv \bar{C} \bmod \mathfrak{p}$. For $\bar{C}$, the zero set of $q_{\bar{R}}$ is $\Theta(\bar{\mathbb{K}}) \pm \bar{R}$, see \cite{grant1990formal}. Thus, the points in $\Theta \pm R$ are zeros of $q_R$.
We show that 
 \begin{equation}\label{eq:subvariety}
 \{q_R=0\} \cap \U     
 \end{equation}
 has dimension one. As $\U$ is irreducible, it is enough to show that \eqref{eq:subvariety} is a proper subset. Indeed, suppose that $q_R$ vanishes on $U$. Let $Q_1,Q_2\in U(\Fq)$ corresponding to the pairs of points $(P_1,P_2)$ and $(P_1, -P_2)$ with
 \begin{equation}\label{eq:yneq0}
 P_i=(x_i,y_i),\quad  y_i\neq 0,\quad  i=1,2. 
 \end{equation}
 Then their first component in the Mumford representation are the same, but with different second components.
 
 Substituting to $q_R$, we get by the assumption that
 $$
 q_R(Q_1)=q_R(Q_2)=0
 $$
 which yields by \eqref{eq:repr} and \eqref{eq:q} that $y_1y_2=-y_1y_2$ which contradicts \eqref{eq:yneq0}. 
 
 In order to show \eqref{eq:cardinality_zeros_q_1}
 for $R\neq \pm R'$, $R,R'\in \U$, consider $\U$ as embedded in $\mathbb{A}^8$ with respect to the coordinates $z_{11},z_{12},z_{22},z_{111},z_{112},z_{122},z_{222}, z $. Then for all $R\in \U$, $q_R=0$ defines a seven dimensional affine hyperspace in $\mathbb{A}^8$. Thus the hyperspaces corresponding to $R$ and $R'$ are the same if
 $$
 z_{11}(R)=z_{11}(R'), \quad  z_{12}(R)=z_{12}(R') \quad \text{and} \quad  z_{22}(R)=z_{22}(R') 
 $$
 and thus $R = \pm R'$.

Consider the solutions
 $$
 q_R(z_{11},z_{12}, z_{22})=q_{R'}(z_{11},z_{12}, z_{22})=0, \quad $$
 where $(z_{11},z_{12}, z_{22},z_{111} ,z_{112},z_{122},z_{222}, z)\in \U$.
 As $R\neq \pm R'$, the hyperplanes defined by  $q_R$ and $q_{R'}$ are distinct. Moreover, the linear equation system
 \begin{align*}
 q_R(z_{11},z_{12}, z_{22})= z_{11}-z_{11}(R)+z_{12}z_{22}(R)-z_{12}(R)z_{22}&=0,\\
 q_{R'}(z_{11},z_{12}, z_{22})= z_{11}-z_{11}(R')+z_{12}z_{22}(R')-z_{12}(R')z_{22}&=0
 \end{align*}
 has full rank and can be transformed to the form
 \begin{equation}\label{eq:LES}
  \begin{pmatrix}
  0 & a_{11} & a_{12}\\
  1 & a_{21} & a_{22}
  \end{pmatrix}
  \begin{pmatrix}
  z_{11}\\ z_{12}\\ z_{22}
  \end{pmatrix}
  =
  \begin{pmatrix}
  c_1\\ c_2
  \end{pmatrix}
 \quad \text{with} \quad
 (a_{11},a_{12})
 %\det \begin{pmatrix}
 %  a_{11} & a_{12}\\
 %  a_{21} & a_{22}
 % \end{pmatrix}
  \neq (0,0).
 \end{equation}
 As $(z_{11},z_{12}, z_{22},z_{111} ,z_{112},z_{122},z_{222}, z)\in \U$, we can write
 \begin{align}\label{eq:z11}
     z_{11}= \frac{z_{22}z_{12}^2+2b_1 z_{12}^2-b_2z_{22}z_{12}-2b_3z_{12}+b_4z_{22}+2b_5-2y_1y_2 }{z_{22}^2+4z_{12}},
 \end{align}
 where $y_1, y_2$ are the $y$-coordinate of the points $(x_1,y_1),(x_2,y_2)\in C$ such that
 \begin{equation}\label{eq:z}
x_1x_2=-z_{12} \quad \text{and} \quad  x_1+x_2=z_{22}.
\end{equation}
 Substituting \eqref{eq:LES} into \eqref{eq:z11}, we obtain
 \begin{align*}
     2y_1y_2=&
     (z_{22}^2+4z_{12})\left(a_{21}z_{12} +a_{22}z_{22}-c_2\right)
     +z_{22}z_{12}^2+2b_1 z_{12}^2-b_2z_{22}z_{12}\\
     &-2b_3z_{12}
     +b_4z_{22}+2b_5.
 \end{align*}
Taking the square and substituting $y_i^2=f(x_i)$, ($i=1,2$), we get

\begin{align}\label{eq:sym}
     4 f(x_1)f(x_2)=
     &\Big((z_{22}^2+4z_{12})\left(a_{21}z_{12} +a_{22}z_{22}-c_2\right)
     +z_{22}z_{12}^2+2b_1 z_{12}^2\\ \notag
     &-b_2z_{22}z_{12}
     -2b_3z_{12}+b_4z_{22}+2b_5\Big)^2.
 \end{align}
 
Assume, that $a_{11}\neq 0$. The case $a_{11}=0$ and $a_{12}\neq 0$ can be handled similarly and we leave it to the reader.
 
Write
$$
g(z_{12},z_{22})= z_{22}^2 \left(a_{21}z_{12} +a_{22}z_{22}\right)+z_{22}z_{12}^2.
$$
 
First, we consider the case that
\begin{equation}\label{eq:g}
g(z_{12},z_{22})=g\left(\frac{c_1-a_{12}z_{22}}{a_{11}},z_{22}\right)
\end{equation}
has degree 3. As $f(x_1)f(x_2)$ is a symmetric polynomial in $x_1$ and $x_2$, by \eqref{eq:z} we can write
$$
f(x_1)f(x_2)=F(z_{12},z_{22}), \quad F\in \F_q[z_{12},z_{22}], \quad \deg F\leq 5.
$$
Hence, after the substitution $z_{12}=(c_1-a_{12}z_{22})/a_{11}$, the polynomial equation~\eqref{eq:sym} is non-trivial with degree at most 6. Thus, 
by~\eqref{eq:LES}, we obtain at most 12 possible solutions for $(z_{11},z_{12}, z_{22},z_{111} ,z_{112},z_{122},z_{222}, z)$.
 
Now consider the case, when~\eqref{eq:g} has degree at most two. By \eqref{eq:z} and \eqref{eq:x2} we have
\begin{equation}\label{eq:x2}
x_2=\frac{c_1-a_{12}x_1}{a_{12}-a_{11}x_1}.
\end{equation}
Then substituting it into~\eqref{eq:sym}, after clearing the denominator, we get that the left-hand side has degree 10 while the right-hand side has degree at most 9. Thus, there are at most $10$ solutions for $x_1$, and therefore at most $20$ solutions for $$(z_{11},z_{12}, z_{22},z_{111} ,z_{112},z_{122},z_{222}, z)$$ by \eqref{eq:LES}, \eqref{eq:z}  and \eqref{eq:x2} which proves the result.
\end{proof}

\begin{lemma}\label{lemma:intersection}
For $D \in U(\Fq)$ we have
$$
|\{\Theta(\Fq) + D\} \cap \Theta(\Fq)| \leq 2.
$$
\end{lemma}
\begin{proof}
We want to count the number of elements $\mathcal{P}_i \in \Theta(\Fq)$ such that 
$$
\mathcal{P}_i + D = \mathcal{P}_j \quad \text{for some~} \mathcal{P}_j \in \Theta(\Fq).
$$
Since $D \in  U(\Fq)$, we have
$$
\psi(D) = \begin{cases}
P + Q - 2\cO, &P,Q \in C(\Fq), P \neq Q, \\
2P - 2\cO, &P \in C(\Fq), \\
P' + Q' - 2\cO, &P',Q' \in C(\F_{q^2}), P', Q' \text{~are conjugates over~} \Fq. 
\end{cases}
$$
If $\psi(D) = P + Q - 2\cO, P,Q \in C(\Fq), P \neq Q$, then for 
$$
\psi(\mathcal{P}_i) = \begin{cases}
-P -\cO, &\psi(\mathcal{P}_i + D) = Q -\cO \in \Theta(\Fq) \\
-Q -\cO, &\psi(\mathcal{P}_i + D) = P -\cO \in \Theta(\Fq) ,
\end{cases}
$$
we get 2 intersection points.

If $\psi(D) = 2P - 2\cO, P \in C(\Fq)$, then for $\psi(\mathcal{P}_i) = -P-\cO$, we get $\psi(\mathcal{P}_i + D) = P -\cO \in \Theta(\Fq)$. We only get one intersection point.

If $\psi(D) = P' + Q' - 2\cO, P',Q' \in C(\F_{q^2})$ and $P', Q' \text{~are conjugates over~} \Fq$, we get no intersection points, since if $\mathcal{P}_i + D = \mathcal{P}_j, $ where $\mathcal{P}_i, \mathcal{P}_j \in \Theta(\Fq)$, with $\psi(\mathcal{P}_i) = Q_i - \cO$ and $\psi(\mathcal{P}_j) = Q_j - \cO, Q_i, Q_j \in C(\Fq)$ we get 
$$
\psi(D) = P' + Q' - 2\cO = \psi(\mathcal{P}_j + (-\mathcal{P}_i)) = Q_j - Q_i - 2\cO
$$
which is a contradiction.
\end{proof}

\begin{prop}\label{prop:main_1}
	Let $ h \in \Fq(U)$ be a rational function with pole divisor of the form $ n\Theta$, for $n \geq 1$. Let $ W_0 \in J_C(\Fq) $ and $ D \in J_C(\Fq) $ be an element of order $ t $. Let $L$ be a positive integer with 
	\begin{equation}\label{eq:prop-L}
	L < \min\left\{ \frac{t-1}{2} - |\Theta(\Fq)| , \frac{|\Theta(\Fq)| - 2}{20}\right\}. 
	\end{equation}
	Let $ j_0< j_1< \dots< j_{L} \leq L + |\Theta(\Fq)|+1$ be positive integers such that $j_iD + W_0 \not \in \Theta(\Fq)$. Let $ c_0, \dots , c_L  \in \Fq$ with $ c_L \neq 0 $. Then the rational function $H \in \Fq(U)$, with 
	\begin{equation*}
	H(Q) = \sum_{l=0}^{L} c_l h(Q + j_l D + W_0)
	\end{equation*}
	is non-constant and has degree 
	$$ 
	\deg H \leq 6(L +1) \deg h.
	$$
\end{prop}
\begin{proof}
First, we consider the case when $ W_0 = 0 $. Defining the function $ h_D : Q \mapsto h(Q + D) $ yields 
$$
H(Q) = \sum_{l=0}^{L-1} c_l h_{j_lD}(Q) + c_L h_{j_LD}(Q).
$$ 
We show that there exists $ Q \in \U $ such that it is a pole of $ h_{j_LD} $, but not a pole of any other terms $ h_{j_lD},$ for $ j_l<j_L $.

Observe that $h_{j_LD} $ has a pole at $ Q $ when $ Q \in \Theta(\Fq) - j_LD $.
From \cref{lem:1}, we know that $ \Theta(\Fq) -j_LD \subseteq \{q_{j_LD} = 0\} $.   Moreover, it follows from  \cref{lemma:intersection}, that
\begin{equation}\label{eq:card_j_L}
    |\Theta(\Fq) -j_LD|\geq |\Theta(\Fq)|-2.
\end{equation}

By \eqref{eq:prop-L}, we have
$$
   j_l + j_L \leq j_{L-1} + j_L \leq 2L + 2|\Theta(\Fq)| + 1 < t,
$$
for any $l< L$, whence $ j_lD \neq -j_LD $ for $ j_l<j_L $. 
Then by Lemma \ref{lem:1}, we obtain
$$
\left| \Big((\Theta(\Fq) - j_LD) \cap \U \Big) \cap \{q_{j_lD} = 0\} \right|\leq
\left| \{q_{j_LD}=0\} \cap \{q_{j_lD} = 0\} \cap \U\right|\leq 20 .
$$
Thus, by \eqref{eq:card_j_L} we have that
\begin{align}\label{eq:set}
		 &\left|\Big((\Theta(\Fq) - j_LD) \cap \U \Big) \setminus \left( \bigcup\limits_{l=0}^{L-1} \{q_{j_lD} = 0 \} \right)\right| \\ \notag
		 =&\left| \bigcap\limits_{l=0}^{L-1} \bigg( \Big((\Theta(\Fq) - j_LD) \cap \U \Big)  \setminus \{q_{j_lD} = 0\} \bigg)\right| \geq |\Theta(\Fq)|-2-20L.
\end{align}
By \eqref{eq:prop-L}, the set in \eqref{eq:set} is non-empty, and thus there exists a point $Q$ which is a pole of $h_{j_LD}$ but not a pole of any other term of $H$. Hence, $H$ is non-constant.

In the case when $ W_0 \neq 0 $, we can define 
$$
\widetilde{H}(Q) = H(Q-W_0) = \sum_{l=0}^{L} c_lh(Q + j_l D + W_0-W_0).
$$
Then, by the case $\widetilde{W}_0 = W_0 - W_0 = 0$, we obtain that $ \widetilde{H}$ is non-constant. 
Therefore, there exist $ Q_1,Q_2 \in \U $ such that $ \widetilde{H}(Q_1) \neq \widetilde{H}(Q_2) $. Hence, we obtain
$$
	H(Q_1-W_0) \neq H(Q_2 - W_0).
$$
This proves that $ H $ is non-constant.

To estimate the degree of $ H $, we first estimate the degree of the functions $ h_{j_lD}$. Let $l$ be arbitrary, define $R = j_l D + W_0$, 
Let $ z_{ij}^{R}, z_{ijk}^{R}, z^{R} $ be as defined in \eqref{eq:add_fn}, then we can write $ h_{R}(Q) $ as 
\begin{equation*}
	h_{R}(Q) =h\big(Q + R\big) =h\big(z_{11}^{R}(Q), \dots , z_{222}^{R}(Q),z^{R}(Q)\big).
\end{equation*}
It follows from \eqref{eq:add_deg} that
\begin{align*}
	\deg h_{R} \leq (\deg h) \big(\max \{ \deg z_{ij}^{R}, \deg z_{ijk}^{R}, \deg z^R \}\big) = 6 \deg h,
\end{align*}
and thus
\begin{equation*}
	\deg H \leq \deg \bigg( \sum_{l=0}^{L} c_l h_{R} \bigg) \leq 6(L+1) (\deg h).
\end{equation*}
In particular, the degree does not depend on $ W_0 $.

\end{proof}

\section{Further preliminaries}\label{sec:preliminaries}

\subsection{Linear complexity}

We need the following result on linear complexity, see \cite[Lemma~ 6]{langeCertainExponentialSums2005}.
\begin{lemma}\label{lemma:linComp}
Let $(s_n)$ be a linear recurrent sequence of order $L$ over any finite field $\Fq$ defined by a linear recursion
$$
s_{n+L}=c_0s_n +\dots + c_{L-1}s_{n+L-1}, \quad n\geq 0.
$$
Then for any $T\geq L+1$ and pairwise distinct positive integers $j_1,\dots, j_{T}$, there exist $a_1,\dots, a_T \in \Fq$, not all equal to zero, such that
$$
\sum_{i=1}^Ta_is_{n+j_i}=0, \quad n\geq 0.
$$
\end{lemma}

\subsection{Vanishing loci of polynomials}

Let $\K$ be a field and for polynomials $f_1,\dots, f_k\in \K[\mathbf{X}]$ with $\mathbf{X}=(X_1,\dots, X_n)$. We denote
$$
V_\K(f_1,\dots, f_k)=\{\mathbf{x}\in\K^n: f_1(\mathbf{x})=\dots=f_k(\mathbf{x})=0 \} 
$$
and
$$
V(f_1,\dots, f_k)=V_{\overline{\K}}(f_1,\dots, f_k).
$$
The following result is a multidimensional version of B\'ezout's Theorem, see
\cite[Theorem~3.1]{schmid1995affine}.

\begin{lemma}\label{lemma:bezout}
Assume, that  $f_1,\dots, f_k\in \K[\mathbf{X}]$ and $\dim V(f_1,\dots, f_k) \leq 0$. Then
$$
| V(f_1,\dots, f_k)|\leq \prod_{i=1}^k \deg f_i.
$$
\end{lemma}

\section{Main results}\label{sec:main}

Recall, that for $D\in J_C(\Fq)$ we have defined the sequence $(W_n)$ recursively by \eqref{eq:S}, namely
$$
W_n=D + W_{n-1}=nD+ W_0, \quad n=1,2,\dots,
$$
with some initial value $W_0\in J_C(\Fq)$. We can assume, that $D,W_0\in J_C(\Fq)$ and thus  all sequence elements are defined over $\Fq$. 
Let $h\in \Fq(J_C)$ and consider the sequence
\begin{equation}\label{eq:w}
w_n=h(W_n), \quad n=0,1,\dots
\end{equation}
with the convention that, if $h$ is not defined at $W_n$, we set $w_n=0$. Clearly, $(W_n)$ and $(w_n)$, are purely periodic sequences, and if $t$ is the order of $D$ in $J_C(\Fq)$, then $t$ is the period length of $(W_n)$. However, $(w_n)$ may have a smaller period length.

\begin{theorem}\label{thm:main}
	Let $ C $ be a hyperelliptic curve of genus $ 2 $. Let $ h \in \Fq(U) $ be a rational function in the function field of the Jacobian with pole divisor of the form $ n\Theta$, for $n \geq 1$.  If $W_0 \in J_C(\Fq)$ and $ D \in J_C(\Fq) $ is of order $ t $ and $ w_n$ is defined by \eqref{eq:w}, then 
	\begin{equation*}
		L(w_n,N) \geq \left\lfloor c 
		   \frac{\min\{t,N\}}{ q \deg h }\right\rfloor
	\end{equation*}
	for some absolute constant $c>0$. 
\end{theorem}

Theorem~\ref{thm:main} yields a lower bound on the linear complexity of the components \eqref{eq:coordinates} in the Mumford representation of $(W_n)$. Clearly, $u_2$ is constant on $\U$. However, applying the result for $h=-z_{22}, -z_{12}, z_{122}$ and $z_{222}$ and observing that $\Theta$ is a pole of them (cf. \eqref{eq:iota}), we get a lower bound on the linear complexity of the components $u_1,u_0, v_1$ and $v_0$ by \eqref{eq:repr}. As all functions $-z_{22}, -z_{12}, z_{122}$ and $z_{222}$ have degree one, we get \eqref{eq:example}.

The result is non-trivial if $t>c q$ for some constant which may depend on $\deg h$. However, the most important case is that $t$ is close to $|J_C(\Fq)|$ which is $(1+o(1))q^2$ by \eqref{eq:size_of_jac}.

\begin{proof}
Let $\Theta(\Fq)$ be defined as in \eqref{eq:Theta}. Then by \eqref{eq:number_of_points} we have $|\Theta(\Fq)|=|C(\Fq)|=q+O(q^{1/2})$. We can assume
\begin{equation}\label{eq:assumption2}
    N \geq 6|\Theta(\Fq)|+2q +3
\end{equation}
and
\begin{equation}\label{eq:assumption}
    L<\min\left\{\frac{t-2-5|\Theta(\Fq)|-2q}{2|\Theta(\Fq)|}, \frac{N-3-6|\Theta(\Fq)|-2q}{2|\Theta(\Fq)|+1}, \frac{t-1}{2}-|\Theta(\Fq)|, \frac{|\Theta(\Fq)|-2}{20} \right\}
\end{equation}
since otherwise, we can choose the absolute constant $c$ small enough so that the theorem holds trivially.

Let $L$ be the $N$th linear complexity of the sequence $(w_n)$ and let $c_0,\dots, c_L\in \Fq$ such that
$$
w_{n+L}=c_0w_n+\dots+c_{L-1}w_{n+L-1}, \quad 0\leq n\leq N-L-1.
$$
Let $(s_n)$ be the infinite linear recurrent sequence with $s_n=w_n$ for $0\leq n<N$ and 
$$
s_{n+L}=c_0s_n+\dots+c_{L-1}s_{n+L-1}, \quad n \geq 0.
$$
Let 
\begin{equation}\label{eq:j_L}
1\leq j_0<\dots<j_{L}\leq |\Theta(\Fq)|+L+1
\end{equation}
be the smallest integers such that $W_{j_i}\not\in\Theta, 0 \leq i \leq L $. 
Then by \cref{lemma:linComp}, there exist $a_0\dots, a_{L}\in\Fq$ not all equal to zero, such that
$$
\sum_{i=0}^La_is_{n+j_i}=0, \quad n\geq 0,
$$
whence, 
\begin{equation}\label{eq:rec}
\sum_{i=0}^{L}a_iw_{n+j_i}=0, \quad 0\leq n <  \min \left\{ N-j_L, t \right\} 
\end{equation}
as $w_n=s_n$ for $0\leq n\leq N$.

We define 
\begin{equation}\label{eq:T}
T = \min \left\{ N-j_L, t \right\}
\end{equation}
and put
$$
\mathcal{N} = \Big\{ 0 \leq n<T : \ nD,  nD \pm (j_i D + W_0)\not\in \Theta(\Fq), {~\text{for}~} 0\leq i \leq L\Big\}.
$$ 

We observe that there are at most $|\Theta(\Fq)|$ points $nD\in\Theta(\Fq)$ and similarly for each $j=\pm j_i$ ($i=0,\dots, L$), there are  $2|\Theta(\Fq)|$ elements $nD \pm (j_iD + W_0) \in\Theta(\Fq)$. Hence, by \eqref{eq:assumption} we have,
\begin{equation}\label{eq:N}
|\mathcal{N}| \geq T - 2L|\Theta(\Fq)| - 3|\Theta(\Fq)|>2|\Theta(\Fq)|+2(q+1).
\end{equation}

Define
\begin{equation}\label{min_zeros_F}
	H(Q) = \sum_{i=0}^{L} a_{i} h( Q+ j_{i}D +W_0). 
\end{equation}
For $Q=nD$, $n\in\cN$, using \cref{lemma:addition_grant}, we see that $H$ is well-defined.
By \eqref{eq:rec}, $H$ vanishes on $Q=nD$, $n\in\cN$. We give an upper bound on the number of zeros of $H$ to get the results together with \eqref{eq:T} and \eqref{eq:N}.

Let us fix a representation of $H$ as a rational function $G_1/G_2\in\Fq(\mathbf{\X})$. Then this set is finite and contains the zeros of $H$ and thus
\begin{equation}\label{eq:N_V}
    |\cN|\leq |  V_{\Fq}(f_2,\dots,f_7,G_1 )|
\end{equation}
where $f_2, \dots, f_7 $ are the defining equations of $\U$. See \cref{sec:def_eq_J}.

In order to estimate the size of \eqref{eq:N_V}, for $r\in\Fq$, write $g_r(\mathbf{\X})=\X_{12}-r$. Clearly,
\begin{equation}\label{eq:slicing}
  V_{\Fq}(f_2,\dots,f_7,G_1)\subset \bigcup_{r\in\Fq}   V(f_2,\dots,f_7,G_1, g_r).
\end{equation}
We claim that for any $r\in\Fq$,
\begin{equation}\label{eq:finite_dim}
    \dim(V(f_2,\dots,f_7,G_1,g_r)) = 0.
\end{equation}
From \cref{prop:main_1}, we know that $H$ is non-constant, therefore,
$$
\dim(V(f_2,\dots,f_7,G_1)) = 1.
$$
Hence, if \eqref{eq:finite_dim} were not true, then $\dim(V(f_2,\dots,f_7,G_1,g_r)) = 1$ and so
$$
g_r \in \langle f_2,\dots,f_7,G_1 \rangle.
$$
Let $D \in \U$ such that $G_1(D) = 0$. Then we also have $g_r(D) = 0$. As all $nD$ for $n\in\cN$ are zeros of $G_1$, it has at least 
\begin{equation}\label{eq:claim}
    2|\Theta(\Fq)|+2(q+1) + 1
\end{equation}
zeros by \eqref{eq:N}.
On the other hand, if $D=[P+Q-2\cO]\in U(\Fq)$, with $P=(x_P,y_P), Q=(x_Q,y_Q)$, is a zero of $g_r$, then
$$
g_r(D)= z_{12}(D)-r=-x_Px_Q-r=0.
$$
If $P,Q\in C(\Fq)$, then $P$ determines $Q$ apart from sign, therefore there are at most $2|\Theta(\Fq)|$ such zeros. On the other hand, if $P,Q\in C(\F_{q^2})$ and $P$ and $Q$ are conjugated, then we must have $x_Q=x_P^q$ and thus $x_P^{q+1}=-r$. As there are at most $q+1$ solutions, there are at most $2(q+1)$ such divisors. Then it yields that $g_r$ has at most $2|\Theta(\Fq)|+2(q+1)$ zeros over $\Fq$, which contradicts \eqref{eq:claim}. Then we have proved \eqref{eq:finite_dim}.

By Lemma~\ref{lemma:bezout} and \eqref{eq:finite_dim} we have
$$
|V(f_2,\dots,f_7,G_1,g_r)|\leq \prod_{i=2}^7 \deg f_i \cdot \deg G_1\leq 216 \deg G_1, 
$$
and thus by \eqref{eq:slicing},
\begin{equation}
       V_{\Fq}(f_2,\dots,f_7,G_1)\leq  216 q\deg G_1\leq 216 q\deg H. 
\end{equation}
Then it follows from Proposition~\ref{prop:main_1}, \eqref{eq:N} and \eqref{eq:N_V} that
\begin{align*}
T-2L|\Theta(\Fq)|-3|\Theta(\Fq)| \leq |\cN| \leq  216 q \deg H \leq 1296(L+1)q \deg h. 
\end{align*}
Whence \eqref{eq:j_L} and \eqref{eq:T} yields the result with some absolute constant $c>0$.
\end{proof}

\section{Comments}
The usage of elliptic curves in pseudorandom number generation is an extensive research direction,
see for example the survey paper \cite{survey_Shparlinski} for a discussion of properties of pseudorandomness of the elliptic curve case of the sequence \eqref{eq:S} and further references as well as other constructions.

Despite some results on the application of hyperelliptic curves in pseudorandom number generation (see \cite{ShparlinskiLange05,Rezaeian,Reza2}) this
line of research has never been studied systematically. In Theorem~\ref{thm:main}, we obtain lower bounds on the linear complexity of the sequence derived from~\eqref{eq:S}. 

The bound we obtain is non-trivial and sufficient for application, however we conjecture that stronger bound for the $N$-th linear complexity holds. Numerical experiments suggest that if $h$ is a coordinate function in \eqref{eq:coordinates}, then for most cases $L(w_n,t)=\lceil t/2 \rceil$.

Additional to the $N$-th linear complexity,
other pseudorandom properties, like uniform distribution, also need to be investigated. The main tool would be the higher genus analogue of \cite{KohelShparlinski} where the authors obtained bounds on exponential sums over elliptic curves.

 \section*{Acknowledgement}
The authors wish to thank Arne Winterhof for the valuable discussions and Igor Shparlinski for useful comments.
The authors were supported by the Austrian Science Fund
Project P31762.

\appendix
\section{}
\subsection{Defining equations of the Jacobian}
\label{sec:def_eq_J}
Let 
$$
S= \K[\X_0, \X_{11},\X_{12},\X_{22},\X_{111},\X_{112},\X_{122},\X_{222},\X] $$ 
be a polynomial ring over field $\K$, with characteristic $p \neq 2$. Following \cite{grant1990formal}, in particular Theorem 2.5, Theorem 2.11 and Corollary 2.15, we define $f_i $ as follows: 
\begin{align*}
    f_1 = &\X^2 + \X_{11}^2\X_{12} + b_1\X_{11}^2\X_{22} +   b_2\X_{11}^2\X_{12}\X_{22} - b_3\X_{11}\X_{22}^2 + b_4\X_{12}\X_{22}^2 \\
    &-b_5\X_{22}^3 + 2b_1\X\X_{11} -2b_2\X\X_{12} + 2b_3\X\X_{22}+ (b_3-b_1b_2)\X_{11}\X_{12} \\
    &+(b_2^2-b_1b_3)\X_{11}\X_{22} + (b_1b_4-b_2b_3-b_5)\X_{12}\X_{22}-b_1b_5\X_{22}^2 \\
    & +2(b_1b_3 -b_2^2)\X + (b_1b_4-b_5)\X_{11} +b_2(b_2^2 - b_1b_3 )\X_{12} \\
    &(b_3b_4 - b_2b_5)\X_{22} + b_1b_3b_4 - b_2^2b_4 - b_3b_5, \\
    f_2 =&2\X - \X_{11}\X_{22} + \X_{12}^2 - b_2\X_{12} + b_4, \\
    f_3 =&\X_{112} -\X_{222}\X_{12} + \X_{122}\X_{22}, \\
    f_4 =&\X_{111} + \X_{222}\X_{11} + \X_{122}\X_{12} - 2\X_{112}\X_{22} - 2b_1\X_{112} + b_2\X_{122}, \\
    f_5 =&\X_{122}^2 - \X_{11}\X_{22}^2 + 2\X\X_{22} + \X_{11}\X_{12} - b_1\X_{11}\X_{22} - b_2\X_{12}\X_{22} \\
    & +2b_1\X - b_1b_2\X_{12} + b_4\X_{22} + b_1b_4 - b_5 , \\
    f_6 =& \X_{222}^2 - \X_{22}^3 - \X_{12}\X_{22} - b_1\X_{22}^2 - \X_{11} - b_2\X_{22} - b_3 , \\
    f_7 =&\X_{122}\X_{222} - \X_{12}\X_{22}^2 + \X -b_2\X_{12} - b_1\X_{12}\X_{22}, \\
    f_8 =&\X_{111}^2 - \X_{11}^3 - b_3\X_{11}^2 -b_4\X_{11}\X_{12} + 3b_5\X_{11}\X_{22} + 2b_5\X \\
    &+(4b_1b_5 - b_2b_4)\X_{11} - 3b_2b_5\X_{12} + (4b_3b_5-b_4^2)\X_{22} \\
    &4b_1b_3b_5 + b_4b_5 - b_1b_4^2 - b_2^2b_5, \\
    f_9 =& -\X_{111}\X_{112} + b_1\X_{111}\X_{122} -b_2\X_{112}\X_{122} + b_3\X_{112}\X_{222} \\
    &-b_4\X_{122}\X_{222} + b_5\X_{222}^2 -\X^2 -b_1 \X \X_{11} + b_2 \X \X_{12} - b_3 \X \X_{22} \\
    &-b_3\X_{11}\X_{12} + b_1b_3\X_{11}\X_{22} - (b_5 + b_1b_4)\X_{12}\X_{22} + 2b_1b_5\X_{22}^2 \\
    & -2(b_1b_3 + b_4)\X  + (2b_2b_4 + b_1b_2b_3 +b_1b_5 - b_3^2 -b_1^2b_4)\X_{12} \\
    &- 2b_5\X_{11} + 2b_5(b_1^2 - b_2)\X_{22} + b_1b_2b_5 - b_1b_3b_4 - 2b_3b_5 , \\
    f_{10} =& \X_{122}^2 - \X_{111}\X_{122} + \X_{11}\X - b_3\X_{11}\X_{22} + 2b_4\X_{12}\X_{22} -3b_5\X_{22}^2 \\
    &+2b_3\X + (b_1b_4 - b_2b_3 - b_5)\X_{12} - 2b_1 b_5 \X_{22} + b_3b_4 - b_2b_5, \\
    f_{11} =&\X_{111}\X_{222} - \X_{112}\X_{122} -2\X\X_{12} + \X_{11}^2 - 2b_1 \X_{11}\X_{12} \\
    &+3b_2 \X_{11}\X_{22} - 2b_3 \X_{12}\X_{22} + b_4\X_{22}^2 -5b_2\X +b_3\X_{11} \\
    &+ (3b_2^2 - 2b_1b_3)\X_{12} + (b_1b_4 - b_5)\X_{22} - 2b_2b_4, \\
    f_{12}=&\X_{122}^2 - \X_{112}\X_{222} + \X_{22}\X + 2\X_{11}\X_{12} - b_1\X_{11}\X_{22} + 2b_1\X \\
    &+(b_3-b_1b_2)\X_{12} + b_1b_4 - b_5, \\
    f_{13} =&\X_{111}\X_{12} -\X_{112}\X_{11} - b_4\X_{122} + 2b_5\X_{222}, \\
    f_{14} =& 2\X_{122}\X_{11} - \X_{112}\X_{12} -\X_{111}\X_{22} - b_2\X_{112} + 2b_3\X_{122} - b_4\X_{222}.
\end{align*}
One can show that $f_1\in \langle f_5,f_6,f_7\rangle$
and the vanishing locus of these polynomials homogenized with respect to the variable $\X_0$ forms a set of defining equations for the Jacobian $J_C$, i.e 
$$
J_C = V(f_2^h,\dots,f_{14}^h) = \{ z \in \mathbb{P}^8(\bar{\K}) : f_i^h(z) = 0, 2 \leq i \leq 14 \}
$$

\subsection{Rational embedding for the Jacobian} \label{app:rational function}
Recall \cref{eq:iota}, where for any $D \in U$ with $ \psi(D) = (x_1,y_1)+ (x_2,y_2)- 2\cO $, the image of $D$ under the embedding $\iota : J_C \rightarrow \mathbb{P}^8$ is given by
\begin{equation*}
    \iota(D) = (1:z_{11}:z_{12}:z_{22}:z_{111}:z_{112}:z_{122}:z_{222}:z)
\end{equation*}
where, $z_{ij}, z_{ijk}$ are rational functions in the coordinates $x_1, x_2, y_1, y_2$ of $\psi(D)$ as shown below. For more details, see \cite[Equation 1.4]{grant1990formal}.  

\begin{align*}
    z_{11}= &\frac{(x_1 + x_2)(x_1x_2)^2 + 2b_1(x_1x_2)^2 + b_2(x_1 + x_2)x_1x_2 + 2b_3x_1x_2}{(x_1-x_2)^2}\\
    & +\frac{b_4(x_1 + x_2) + 2b_5 - 2y_1y_2}{(x_1-x_2)^2},\\
    z_{12} =& -x_1x_2,\\
    z_{22} =& x_1 + x_2, \\
    z_{111} =& \frac{y_2 \Psi(x_1,x_2) - y_1 \Psi(x_2,x_1)}{(x_1 - x_2)^3}, \text{~where,}\\
    \Psi(x_1,x_2) = &4b_5 + b_4(3x_1 + x_2) + 2b_3x_1(x_1 + x_2) + b_2x_1^2(x_1 + 3x_2) \\
    & + 4b_1x_1^3x_2 + x_1^3x_2(3x_1 + x_2),\\
    z_{112} =& \frac{y_1x_2^2 - y_2x_1^2}{x_1- x_2}, \\
    z_{122} =& -\frac{y_1x_2 - y_2x_1}{x_1- x_2}, \\
    z_{222} =& \frac{y_1 - y_2}{x_1- x_2}. \\
    z =& \frac{1}{2}(z_{11}z_{22} - z_{11}^2 + b_2z_{12} - b_4 )
\end{align*} 
We consider $U$ embedded in $\mathbb{A}^8$ by considering the isomorphism given by 
\begin{equation*}
    (1:z_{11}:z_{12}:z_{22}:z_{111}:z_{112}:z_{122}:z_{222}:z) \mapsto (z_{11},z_{12},z_{22},z_{111},z_{112},z_{122},z_{222},z)
\end{equation*}

\subsection{Addition formulas}
\label{sec:addition_formulas}
In \cref{lemma:addition_grant}, we stated the group law on $U$ as given in \cite[Theorem 3.3]{grant1990formal}. This reference gives instructions as to how one could compute $z_{112}(Q+R),z_{122}(Q+R),z_{222}(Q+R)$ but does not state them explicitly in the paper. To estimate the degree of these functions, we computed the formulas for $z_{112}(Q+R),z_{122}(Q+R),z_{222}(Q+R)$ as follows:  
\begin{align*}
     z_{112}(Q+R)= &-\frac{1}{2}z_{112}(Q)-\frac{1}{2}z_{112}(R) + \frac{1}{16}\frac{q_{2}(Q,R)q_{11}(Q,R)}{q(Q,R)^2}\\ &+\frac{1}{8}\frac{q_{1}(Q,R)q_{12}(Q,R)}{q(Q,R)^2} -\frac{1}{16}\frac{q_{112}(Q,R)}{q(Q,R)}
     - \frac{1}{8}\frac{q_{2}(Q,R)(q_{1}(Q,R))^2}{q(Q,R)^3} \\  
     &+ \frac{3}{8}(z_{11}(Q)+z_{11}(R))\frac{q_{2}(Q,R)}{q(Q,R)} + \frac{3}{8}(z_{12}(Q)+z_{12}(R))\frac{q_{1}(Q,R)}{q(Q,R)} \\
     z_{122}(Q+R)= &-\frac{1}{2}z_{122}(Q)-\frac{1}{2}z_{122}(R) + \frac{1}{16}\frac{q_{1}(Q,R)q_{22}(Q,R)}{q(Q,R)^2}\\ &+\frac{1}{8}\frac{q_{2}(Q,R)q_{12}(Q,R)}{q(Q,R)^2} -\frac{1}{16}\frac{q_{122}(Q,R)}{q(Q,R)} \\
     &- \frac{1}{8}\frac{q_{1}(Q,R)(q_{2}(Q,R))^2}{q(Q,R)^3}  
     + \frac{3}{4}(z_{12}(Q)+z_{12}(R))\frac{q_{2}(Q,R)}{q(Q,R)} \\
     z_{222}(Q+R)= &-\frac{1}{2}z_{222}(Q)-\frac{1}{2}z_{222}(R) + \frac{3}{16}\frac{q_{2}(Q,R)q_{22}(Q,R)}{q(Q,R)^2}-\frac{1}{16}\frac{q_{222}(Q,R)}{q(Q,R)} \\
     &- \frac{1}{8} \left( \frac{q_{2}(Q,R)}{q(Q,R)} \right) ^3 
     + \frac{3}{4}(z_{22}(Q)+z_{22}(R))\frac{q_{2}(Q,R)}{q(Q,R)} \\
     z(Q+R) =& \frac{1}{2}(z_{11}(Q+R)z_{22}(Q+R) - z_{11}^2(Q+R) + b_2z_{12}(Q+R) - b_4 )
\end{align*}

To evaluate the addition formulas from \cref{lemma:addition_grant}, we need the following functions:
\newcommand{\p}[2]{z_{#1#2}}
\newcommand{\pp}[3]{2z_{#1#2#3}}
\newcommand{\po}{(2z - b_2 z_{12} + b_4)}

\begin{align*}
    q(Q,R)= &z_{11}(Q)-z_{11}(R)+z_{12}(Q)z_{22}(R)-z_{12}(R)z_{22}(Q), \\
    q_1(Q,R) = &\pp111(Q) - \pp111(R) + \pp112(Q)\p22(R) - \pp112(R)\p22(Q) \\
    & +\pp122(R)\p12(Q) - \pp122(Q)\p12(R), \\ 
    q_2(Q,R) =& \pp112(Q) - \pp112(R) + \pp122(Q)\p22(R) - \pp122(R)\p22(Q) \\
    & +\pp222(R)\p12(Q) - \pp222(Q)\p12(R), \\
    q_{11}(Q,R)=& 4b_3 q(Q,R) + 4b_4(\p12(Q)-\p12(R)) + 4(\po(Q)\p12(R))\\
    &-4(\po(R)\p12(Q))-8b_5(\p22(Q) - \p22(R)) \\ &+2(\pp112(Q)\pp122(R) - \pp112(R)\pp122(Q)),\\
    q_{12}(Q,R) =& 4b_3(\p12(Q)-\p12(R)) + 2b_2(\p12(Q)\p22(R)) \\ &-2b_2(\p12(R)\p22(Q)) - 4(\p11(Q)\p12(R)-\p11(R)\p12(Q))\\
    & +2(\po(Q)\p22(R)-\po(R)\p22(Q)) \\
    & - 2b_4(\p22(Q) - \p22(R)) + \pp222(R)\pp112(Q)-\pp222(Q)\pp112(R), \\
    q_{22}(Q,R) =& 8b_1(\p12(Q)\p22(R)-\p12(R)\p22(Q)) + 4b_2\p12(Q)\\
    & -4b_2\p12(R) -8(\p11(Q)\p22(R)-\p11(R)\p22(Q)) \\
    & - 4(\po(Q)-\po(R)) \\
    & + 2(\pp122(Q)\pp222(R) -\pp122(R)\pp222(Q)),\\
    q_{111}(Q,R) =& 4b_3q_1(Q,R) \\
    & + 4(\pp111(Q)\p22(Q)\p12(R)-\pp111(R)\p22(R)\p12(Q)) \\
    & + \pp122(R)(2\p12(Q)(6\p11(Q) - 2\p11(R)+4b_3)-4b_4\p22(Q))\\
    & - \pp122(Q)(2\p12(R)(6\p11(R) - 2\p11(Q)+4b_3)-4b_4\p22(R))\\
    & + \pp112(Q)(\p12(R)(12\p12(R) - 8\p12(Q)+4b_2)+4b_4)\\
    & - \pp112(R)(\p12(Q)(12\p12(Q) - 8\p12(R)+4b_2)+4b_4)\\
\end{align*}
\subsubsection{Formulas for $q_{ijk}$}
\label{sec:addition_formulas_extended}
    There are multiple ways to compute formulas for $q_{ijk}$, e.g. 
    $$
    q_{ijk} (Q,R) = \mathcal D_i(q_{jk})(Q,R) = \mathcal D_j(q_{ik})(Q,R) = \mathcal D_k(q_{ij})(Q,R)
    $$
    where $\mathcal D_i$ is the differential operator as defined in the proof of \cite[ Theorem 3.3 ]{grant1990formal}. We used the Python package \texttt{SymPy} to compute the formulas and chose the following expressions for $q_{ijk}$:
	\begin{align*}
		q_{111} &= \mathcal D_1(q_{11}), \\
		q_{112} &= \mathcal D_2(q_{11}), \\
		q_{122} &= \mathcal D_2(q_{12}), \\
		q_{222} &= \mathcal D_2(q_{22}).
	\end{align*}

	\begin{align*}
	q_{111}(Q,R) = &~\pp112(R)
	\left(- 12 \p12(Q)^{2} + \p12(Q) \left(8 \p12(R) - 4 b_{2}\right) - 4 \p22(Q) b_{3} - 4 b_{4}
	\right) \\ 
	&+\pp111(R)
	\left(- 4 \p12(Q) \p22(R) - 4 b_{3}
	\right) \\ 
	&+\pp111(Q)
	\left(4 \p12(R) \p22(Q) + 4 b3
	\right) \\ 
	&+\pp112(Q)
	\left(- 8 \p12(Q) \p12(R) + 12 \p12(R)^{2} + 4 \p12(R) b_{2} + 4 \p22(R) b_{3} + 4 b_{4}
	\right) \\ 
	&+\pp122(Q)
	\left(4 \p11(Q) \p12(R) - 12 \p11(R) \p12(R) - 12 \p12(R) b_{3} + 4 \p22(R) b_{4}
	\right) \\ 
	&+\pp122(R)
	\left(12 \p11(Q) \p12(Q) + \p12(Q) \left(- 4 \p11(R) + 12 b_{3}\right) - 4 \p22(Q) b_{4}
	\right) \\
	q_{112}(Q,R) = &~\pp222(Q)
	\left(4 \p11(Q) \p12(R) - 4 \p12(R) b_{3} - 8 b_{5}
	\right) \\ 
	&+\pp112(Q)
	\left(- 4 \p11(R) + 4 \p12(R) \p22(Q) + \p12(R) \left(12 \p22(R) + 8 b_{1}\right) + 4 b_{3}
	\right) \\ 
	&+\pp112(R)
	\left(4 \p11(Q) + \p12(Q) \left(- 12 \p22(Q) - 4 \p22(R) - 8 b_{1}\right) - 4 b_{3}
	\right) \\ 
	&+\pp122(Q)
	(- 8 \p11(R) \p22(R) - 8 \p12(Q) \p12(R) - 4 \p12(R)^{2} - 4 \p12(R) b_{2} \\
	&+ 4 \p22(R) b_{3} + 4 b_{4} ) \\ 
	&+\pp122(R)
	\left(8 \p11(Q) \p22(Q) + 4 \p12(Q)^{2} + \p12(Q) \left(8 \p12(R) + 4 b_{2}\right) \right.\\
    &- \left. 4 \p22(Q) b_{3} - 4 b_{4} \right) \\ 
	&+\pp222(R)
	\left(\p12(Q) \left(- 4 \p11(R) + 4 b_{3}\right) + 8 b_{5}
	\right) \\
	q_{122}(Q,R) = &~\pp112(R)
	\left(- 6 \p22(Q)^{2} + \p22(Q) \left(- 2 \p22(R) - 4 b_{1}\right) - 2 b_{2}
	\right) \\ 
	&+\pp122(R)
	\left(- 4 \p11(Q) + \p22(Q) \left(4 \p12(R) - 2 b_{2}\right) - 4 b_{3}
	\right) \\ 
	&+\pp222(Q)
	\left(2 \p11(Q) \p22(R) - 4 \p11(R) \p22(R) - 2 \p12(R)^{2} - 4 \p12(R) b_{2} - 2 b_{4}
	\right) \\ 
	&+\pp112(Q)
	\left(2 \p22(Q) \p22(R) + 6 \p22(R)^{2} + 4 \p22(R) b_{1} + 2 b_{2}
	\right) \\ 
	&+\pp222(R)
	\left(4 \p11(Q) \p22(Q) - 2 \p11(R) \p22(Q) + 2 \p12(Q)^{2} + 4 \p12(Q) b_{2} + 2 b_{4}
	\right) \\ 
	&+\pp122(Q)
	\left(4 \p11(R) - 4 \p12(Q) \p22(R) + 2 \p22(R) b_{2} + 4 b_{3}
	\right) \\
	q_{222}(Q,R) = &~\pp222(R)
	\left(- 12 \p11(Q) + 4 \p11(R) + \p12(Q) \left(12 \p22(Q) + 16 b_{1}\right)
	\right) \\ 
	&+\pp122(R)
	\left(- 8 \p12(Q) - 8 \p12(R) - 12 \p22(Q)^{2} - 16 \p22(Q) b_{1} - 8 b_{2}
	\right) \\ 
	&+\pp112(Q)
	\left(- 4 \p22(Q) - 8 \p22(R)
	\right) \\ 
	&+\pp222(Q)
	\left(- 4 \p11(Q) + 12 \p11(R) + \p12(R) \left(- 12 \p22(R) - 16 b_{1}\right)
	\right) \\ 
	&+\pp112(R)
	\left(8 \p22(Q) + 4 \p22(R)
	\right) \\ 
	&+\pp122(Q)
	\left(8 \p12(Q) + 8 \p12(R) + 12 \p22(R)^{2} + 16 \p22(R) b_{1} + 8 b_{2}
	\right)
	\end{align*}

\bibliographystyle{amsplain}
\bibliography{Bibliography.bib}
\end{document}